\documentclass[11pt]{amsart}
\usepackage[parfill]{parskip}    
\usepackage{fullpage}
\usepackage{graphicx}
\usepackage{amssymb,amsmath,latexsym}
\usepackage{tikz}
\usetikzlibrary{arrows,decorations.markings,decorations.pathreplacing,calc,matrix,intersections}
\tikzset{->-/.style={decoration={markings,mark=at position #1 with {\arrow{>}}},postaction={decorate}}}

\usepackage{epstopdf}
\usepackage{ifpdf}
\usepackage{color}
\usepackage{float}
\usepackage{amscd,amsmath, mathabx}
\usepackage{amssymb,amsmath,latexsym,color,enumerate,tikz}
\usepackage[all]{xypic}       
\usepackage[varg]{pxfonts}
\usepackage{amsmath,amsthm,amsfonts,amssymb,fancyhdr,graphics,relsize,tikz-cd,mathtools,faktor,tikz,changepage}
\usepackage{graphicx}
\usepackage{placeins}

\usepackage{dsfont}
\definecolor{red}{rgb}{1,0,0} 

 \definecolor{darkgreen}{rgb}{0, .7, 0}

 \definecolor{purple}{rgb}{.7, 0, 1}

\tikzset{mynode/.style={draw,circle,fill=black,inner sep=2pt,outer sep=0.5pt}}

\newtheorem{theorem}{Theorem}[section]
\newtheorem*{theorem*}{Theorem}
\newtheorem*{lemma*}{Lemma}
\newtheorem{proposition}[theorem]{Proposition}
\newtheorem{lemma}[theorem]{Lemma}
\newtheorem{corollary}[theorem]{Corollary}

\theoremstyle{definition}
\newtheorem{definition}[theorem]{Definition}

\theoremstyle{remark}

\usepackage{hyperref}
 
\begin{document}
\title{On the Bieri-Neumann-Strebel-Renz invariants and limit groups over Droms RAAGs}
\author{Dessislava H. Kochloukova, Jone Lopez de Gamiz Zearra}

\maketitle

\begin{abstract} For a group $G$ that is a limit group over Droms RAAGs such that $G$ has trivial center, we show that $\Sigma^1(G) = \emptyset = \Sigma^1(G, \mathbb{Q})$. For a group $H$ that is a finitely presented residually Droms RAAG we calculate $\Sigma^1(H)$ and $\Sigma^2(H)_{dis}$. In addition, we obtain a necessary condition for $[\chi]$ to belong to $\Sigma^n(H)$.
\end{abstract}

\section{Introduction}

In this paper we study the Bieri--Neumann--Strebel--Renz invariants of limit groups over Droms right-angled Artin groups (RAAGs). A RAAG is called a \emph{Droms RAAG} if every finitely generated subgroup is again a RAAG. Droms RAAGs can be characterized as the RAAGs where the defining graph does not contain full subgraphs that are squares or lines of length 3 (see \cite{Droms2}). In particular, finitely generated free groups are Droms RAAGs, so the class of \emph{limit groups over Droms RAAGs} contains the class of \emph{limit groups over free groups}. A group $H$ is a \emph{limit group over Droms RAAGs} if it is finitely generated and \emph{fully residually Droms RAAG}, that is, for every finite subset $S$ of $H$ there is a homomorphism $\varphi \colon H \mapsto G_S$ where $G_S$ is a Droms RAAG and $\varphi$ is injective on $S$. It is easy to see that this is equivalent to the existence of a Droms RAAG $G$ such that all $G_S$ can be chosen to be equal to $G$. In this case we say that $H$ is a limit group over $G$.

The class of limit groups over free groups was studied by O. Kharlampovich, A. Miasnikov, under the name of \emph{fully residually free groups}, and independently by Z. Sela, who fixed the name \emph{limit group}. The class of limit groups over free groups contains all finitely generated free abelian groups and surface groups of Euler characteristic at most $-2$.

 The study of \emph{limit groups over coherent RAAGs}, in particular of limit groups over Droms RAAGs, was initiated by M. Casals-Ruiz, A. Duncan and I. Kazachkov in \cite{Montse} and \cite{Montse2}.
 Limit groups over Droms RAAGs share many properties with limit groups over free groups. For instance, in \cite{Montse} M. Casals-Ruiz, A. Duncan and  I. Kazachkov proved that limit groups over Droms RAAGs are finitely presented and of type $FP_{\infty}$, and that their finitely generated subgroups are again limit groups over Droms RAAGs. In \cite{Jone} J. Lopez de Gamiz Zearra showed that finitely presented subdirect products of limit groups over Droms RAAGs behave as finitely presented subdirect products of limit groups over free groups. In \cite{preprint1} D. Kochloukova and J. Lopez de Gamiz Zearra studied further the structure of subdirect products of type $FP_m$ of limit groups over Droms RAAGs. The starting point in \cite{preprint1} is the fact that limit groups over Droms RAAGs are free-by-(torsion-free nilpotent). In the case of limit groups over free groups this was first observed in \cite{Desi}.
 Some of the results in \cite{preprint1} treat the more general class of limit groups over coherent RAAGs, where their  analytic Betti numbers are calculated.
 
The \emph{Bieri--Neumann-Strebel--Renz invariants} are specific open subsets in the character sphere $S(G)$ of a group $G$. They are a tool to control when a subgroup containing the commutator subgroup is of type $FP_n$ or is finitely presented. The invariants are separated in two groups: the homotopical invariants $\{ \Sigma^n(G) \}_{n}$  and the homological ones $\{ \Sigma^n(G, \mathbb{Z}) \}_n$. In general, whenever the invariants are defined we have the inclusions
  \[\dots \subseteq \Sigma^n(G) \subseteq \Sigma^{ n-1} (G) \subseteq \dots \subseteq \Sigma^1(G) \subseteq S(G)\]
  and
   \[\dots \subseteq \Sigma^n(G, \mathbb{Z}) \subseteq \Sigma^{ n-1} (G, \mathbb{Z}) \subseteq \dots \subseteq \Sigma^1(G, \mathbb{Z}) \subseteq S(G),\]
where $S(G)$ denotes the character sphere of $G$.
The homotopical invariant $\Sigma^n(G)$ is defined only for groups $G$ of homotopical type $F_n$. In general, $\Sigma^n(G, \mathbb{Z})$ is defined for any finitely generated group $G$, but if $\Sigma^n(G, \mathbb{Z}) \not= \emptyset$, then $G$ is of homological type $FP_n$. If $G$ is of homotopical type $F_n$ and $n \geq 2$, it is shown in \cite{Renzthesis} that \[\Sigma^n(G) = \Sigma^2(G) \cap \Sigma^n(G, \mathbb{Z}).\]
   
At first, the invariant $\Sigma^1$ was defined for finitely generated metabelian groups and later, in \cite{B-N-S},  R. Bieri, W.D. Neumann and R. Strebel extended the definition to the class of all finitely generated groups and they showed that the invariant is linked with the Thurston semi-norm in $3$-dimensional topology. The homological invariants were studied by R. Bieri and B. Renz in \cite{B-Renz}, and the fact that they control the homological finiteness properties of a subgroup containing the commutator subgroup was proved together with the fact that the invariants are always open subsets of the character sphere.

In general, the Bieri--Neumann--Strebel--Renz invariants are difficult to compute, but they are known for some classes of groups, including limit groups over free groups (see \cite{Desi}).

For example, the case of the Thompson group $F$ and the generalized Thompson groups $F_{n, \infty}$ was treated by Bieri-Groves-Kochloukova,   Witzel-Zaremsky and M. Zaremsky  in   \cite{B-G-K},  \cite{W-Z} and \cite{Z}. For a finitely generated metabelian group $G$, the structure of the complement of $\Sigma^1(G)$ in the character sphere $S(G)$ as a rationally defined polyhedron was proved by Bieri-Groves in \cite{B-Groves} and this result has numerous applications in tropical geometry.  The case when $G$ is the fundamental group of a compact K\"ahler manifold was studied by T. Delzant in \cite{D} and the case of a free-by-cyclic group was considered by Funke-Kielak and D. Kielak in \cite{B-F} and \cite{DK}. The case when $G$ is a  right-angled Artin group was settled by Meier-Meinert-VanWyk (\cite{Meinert-VanWyk}) but for general Artin groups just some particular cases are known. For example, some Artin groups that are not RAAGs were considered by K. Almeida in \cite{Kisnney} and some even Artin groups were considered by R. Blasco Garc\'ia, J. I. Cogolludo Agust\'in, C. Mart\'inez P\'erez in \cite{Conchita} and by D. Kochloukova in \cite{K}.

In Section \ref{Section 7}, we give a brief introduction to the topic. Generalising the results from \cite{Desi}  and \cite{Desi2}, we compute the invariant $\Sigma^1(H)$ for a finitely presented residually Droms RAAG $H$ and the invariants $\Sigma^m(G)$ and $\Sigma^m(G,\mathbb{Q})$ for a limit group $G$ over Droms RAAGs such that $G$ has trivial center.

\medskip

{\bf Proposition A}
{\it Let $G$ be a limit group over Droms RAAGs such that $G$ has trivial center. Then $\Sigma^n(G) = \Sigma^n(G, \mathbb{Z}) = \Sigma^n (G, \mathbb{Q}) = \emptyset$ for every $n \geq 1.$}

\medskip

We recall that a subgroup $H < L_0 \times L_1 \times \cdots \times L_m$ is a \emph{subdirect product} if $H$ projects surjectively on each $L_i$. The subdirect product is \emph{full} if $H \cap L_i$ is non-trivial for each $i\in \{0,\dots,m\}$. In addition, the full subdirect product $H  < L_0 \times L_1 \times \cdots \times L_m$ is \emph{neat} if $L_0$ is abelian, $H \cap L_0$ has finite index in $L_0$ and  $L_i$ has trivial center for $  1 \leq i \leq m$. By the results of  G. Baumslag, A. Miasnikov and V. Remeslennikov in \cite{Baumslag}, a residually Droms RAAG $G$ can be viewed as a full subdirect product of limit groups over Droms RAAGs, that is, $G < L_0 \times L_1 \times \cdots \times L_m$ is a full subdirect product where $L_0= \mathbb{Z}^n$ for some $n \geq 0$ and $L_1, \dots, L_m$ are limit groups over Droms RAAGs such that each $L_i$ has trivial center for $1 \leq i \leq m$. Thus, $n = 0$ if $G$ has trivial center.

By the results of  M. Bridson, J. Howie, C. F. Miller and H. Short in \cite{Bridson} and \cite{Bridson2}, $H$ is  a finitely presented residually free group if and only if  $H$ is a neat full subdirect product in $L_0 \times L_1 \times \cdots \times L_m$ where each group $L_i$ is a limit group over a free group and $H$ maps virtually surjectively on pairs. In \cite{Jone} J. Lopez de Gamiz Zearra extended this to the case of residually Droms RAAGs, that is, $H$ is a finitely presented residually Droms RAAG if and only if $H$ is a neat full subdirect product in $L_0 \times L_1 \times \cdots \times L_m$ where each group $L_i$ is a limit group over a Droms RAAG and $H$ maps virtually surjectively on pairs.

In our next result we give a necessary condition for a point from the character sphere $S(H)$ to belong to  $\Sigma^m(H,\mathbb{Q})$ for a finitely presented residually Droms RAAG $H$.  For a character $\chi \colon H \mapsto \mathbb{R}$ we set \[ H_{\chi}= \{ h\in H \mid \chi(h) \geq 0\},\] $[\chi]$ is the equivalence class of $\chi$, that is, $[\chi]$ is the set of characters $\mathbb{R}_{>0}\chi$ and \[S(H)= \{ [\chi] \mid \chi \text{ is a character of } H\}.\] Then, by definition,
\[ \Sigma^n(H,\mathbb{Q})= \{ [\chi] \in S(H)  \mid \mathbb{Q} \text{ is of homological type } FP_{n} \text{ over } \mathbb{Q}H_{\chi}\}.\]

\medskip
{\bf Theorem B}
{\it Let  $m \geq 2$, $1 \leq n \leq m$ and $H < L_1 \times \cdots \times L_m$ be a finitely presented full subdirect product of limit groups over Droms RAAGs  $L_1, \dots, L_m$ where each $L_i$ has trivial center. Suppose that $[\chi] \in \Sigma^n(H, \mathbb{Q})$. Then
\begin{equation} \label{monod0} p_{j_1, \dots, j_n}(H_{\chi}) = p_{j_1, \dots, j_n}(H) \quad \text{for all} \quad 1 \leq j_1 < \dots  < j_n \leq m, \end{equation} where $p_{j_1, \dots, j_n}\colon H \mapsto L_{j_1} \times \cdots \times L_{j_n}$ is  the canonical projection.}

\medskip
Note that whenever $\Sigma^n(H,\mathbb{Q})\neq \emptyset$, the group $H$ is of homological type $FP_n(\mathbb{Q})$ and by \cite{preprint1}, this implies that $[L_{j_1}\times \cdots \times L_{j_n} \colon p_{j_1,\dots,j_n}(H)]< \infty.$

We recall that for a subgroup $M$ of a group $H$, \[S(H, M) = \{ [\chi] \in S(H) ~| ~\chi(M) = 0 \}.\]

{\bf Corollary C}
{\it a) Let $H < L_1 \times \cdots \times L_m$ be a finitely presented full subdirect product of limit groups over Droms RAAGs  $L_1, \dots, L_m$ where each $L_i$ has trivial center. Then  \[\Sigma^1(H) =  \Sigma^1(H, \mathbb{Q}) = \{ [\chi] \in S(H) \mid p_i(H_{\chi}) = p_i(H) = L_i \text{ for every }1 \leq i \leq m \},\] and thus,
\[S(H) \setminus \Sigma^1(H) = \bigcup_{1 \leq i \leq m} S(H, \ker (p_i)),\]
where $p_i \colon H \mapsto L_i$ is the canonical projection.\\[5pt]
b) If $H$ is a finitely presented residually Droms RAAG, then there exist finitely many subgroups $H_1, \dots, H_m$ of $H$ such that
\[S(H) \setminus \Sigma^1(H) = \bigcup_{1 \leq i \leq m} S(H, H_i).\]
}

Observe that Corollary C shows the converse of Theorem B for $n=1$. Nevertheless, it is still an open problem whether the converse of Theorem B holds for $n\geq 2$ and this conjecture was named by D. Kochloukova and F. Lima in \cite{Desi2} as \emph{the Monoidal Virtual Surjection Conjecture}. In the next result we show that under certain conditions, namely when the Virtual Surjection Conjecture holds, the Monoidal Virtual Surjection Conjecture is true.

The Virtual Surjection Conjecture is a generalization of the VSP Criterion. In \cite{Bridson2}, M. Bridson, J. Howie, C.F. Miller and H. Short proved the VPS Criterion:

{\bf The VSP Criterion}
{\it Let $H$ be a subgroup of the direct product $G_1 \times \cdots \times G_m$ of finitely presented groups $G_1,\dots,G_m$. If $H$ is virtually surjective on pairs, then $H$ is finitely presented.}

\emph{The Virtual Surjection Conjecture} states that the same holds for groups of homotopical type $F_n$, that is, if $n\leq m$ are positive integers, $H$ is a subgroup of the direct product $G_1\times \cdots \times G_m$ where $G_i$ is of type $F_n$ for $1\leq i \leq m$ and $H$ is virtually surjective on $n$-tuples, then $H$ is of type $F_n$.

If $T$ is a subset of $S(H)$, we denote by $T_{dis}$ the set of discrete points $[\chi]$ in $T$, that is, $\chi$ is a discrete character.

\medskip
{\bf Theorem D} {\it a) Let $1 \leq n \leq m$ and $m \geq 2$ be integers and $H < L_1 \times \cdots \times L_m$ be a  full subdirect product of limit groups over Droms RAAGs  $L_1, \dots, L_m$ such that each $L_i$ has trivial center and $H$ is of type $FP_n$ and finitely presented. Suppose that the Virtual Surjection Conjecture holds in dimension $n$. Then \[[\chi] \in \Sigma^n(H, \mathbb{Q})_{dis} = \Sigma^n(H, \mathbb{Z})_{dis} = \Sigma^n(H)_{dis}\] if and only if 	
\begin{equation} \label{monod0} p_{j_1, \dots, j_n}(H_{\chi}) = p_{j_1, \dots, j_n}(H) \quad \text{for all} \quad 1 \leq j_1 < \dots  < j_n \leq m, \end{equation} where $p_{j_1, \dots, j_n}\colon H \mapsto L_{j_1} \times \cdots \times L_{j_n}$ is  the canonical projection. In particular, since the  Virtual Surjection Conjecture holds in dimension $2$, the result holds for $n = 2$ without further assumptions.\\[5pt]
b) If $H$ is  a finitely presented residually  Droms RAAG, then there exist finitely many subgroups $H_{i,j}$ of $H$, where $ 1 \leq i < j \leq n$,  such that
\[S(H)_{dis} \setminus \Sigma^2(H)_{dis} = \bigcup_{1 \leq i < j \leq m} S(H, H_{i,j})_{dis}.\] }

\medskip
Finally, as an application of $\Sigma$-theory, we answer in Corollary \ref{aplication} a question posed in \cite{Jone-Montse} about finiteness properties of finitely presented subgroups of the direct product of two particular RAAGs. 
In \cite{Jone-Montse}, finitely presented subgroups of the direct product of two \emph{coherent RAAGs of dimension $2$} are studied. For that, the authors define two classes of groups, $\mathcal{G}$ and $\mathcal{A}$. The class $\mathcal{G}$ is defined to be the class of cyclic subgroup separable graphs of groups with free abelian vertex groups and infinite cyclic or trivial edge groups. The class $\mathcal{A}$ is the \emph{$Z\ast-$closure} of $\mathcal{G}$. In \cite[Theorem 5.1]{Jone-Montse} it is shown that if $S$ is a finitely presented subgroup of $G_1\times G_2$, where $G_1,G_2\in \mathcal{A}$ and $G_1, G_2$ are finitely generated, then $S$ is virtually the kernel of a homomorphism $f\colon H_1\times H_2 \mapsto \mathbb{R}$, where $H_1,H_2\in \mathcal{A}$ and $\text{im}(f) \simeq \mathbb{Z}^n$ for some $n\in \mathbb{N}\cup \{0\}$. In \cite{Jone-Montse} it was stated that it is expected that Morse theory can be useful to further study the finiteness properties of $S$. Using methods from $\Sigma$-theory we show in Section \ref{sect-f}  that $S$ is of type $F_{\infty}$. Actually, we prove in Proposition E a more general result. Let $\mathcal{C}$ be the class of fundamental groups of finite reduced graphs of groups where the vertex groups are free abelian of rank greater than $1$ and the edge groups are infinite cyclic or trivial.  Let $\mathcal {J}$ be the $Z \ast$-closure of $\mathcal{C}$.

\medskip
{\bf Proposition E} 
{\it Let $S < H_1 \times H_2$ be a co-abelian, finitely presented, subdirect product with $H_1, H_2 \in \mathcal{J}$. Then, $S$ is of type $F_{\infty}$.}

{\bf Acknowledgements}  D. Kochloukova was partially supported by the CNPq grant 301779/2017-1  and by the FAPESP grant 2018/23690-6. J. Lopez de Gamiz Zearra was supported by the Basque Government grant IT974-16 and the Spanish Government grant MTM2017-86802-P.
 
\section{Preliminaries on limit groups over Droms RAAGs}
\label{Limit groups}

Given a simplicial graph $X$, the corresponding \emph{right-angled Artin group (RAAG)}, denoted by $GX$, is generated by the vertex set $V(X)$ of $X$. Then, $GX$ is defined by the presentation
\[ GX= \langle V(X) \mid xy=yx \iff x \text{ and } y \text{ are adjacent}\rangle.\]

In \cite{Droms2}  Carl Droms described when all finitely generated subgroups of a RAAG are again RAAGs, and he showed that $X$ could not contain a square or a line of length $3$. These groups are known as \emph{Droms RAAGs}.

If $G$ is a group, a \emph{limit group over $G$} is a group $H$ that is finitely generated and fully residually $G$, that is, for any finite set of non-trivial elements $T \subseteq H$ there is a homomorphism $\mu \colon H \mapsto G$ which is injective on $T$.

We state several results from \cite{preprint1} about limit groups over Droms RAAGs that generalise earlier results for limit groups over free groups from \cite{Desi}.

\begin{proposition} \cite{preprint1} \label{free-by-(torsion-free nilpotent)}
Limit groups over Droms RAAGs are free-by-(torsion-free nilpotent).
\end{proposition}

\begin{definition}
Let $H$ be a group and let $Z \subseteq H$ be the centralizer of an element. Then, the group $G= H \ast_{Z} (Z \times \mathbb{Z}^n)$ is said to be obtained from $H$ by an \emph{extension of a centralizer}. An \emph{ICE group over $H$} is a group obtained from $H$ by applying finitely many times the extension of a centralizer construction.
\end{definition}

\begin{theorem}[\cite{Montse}] \label{Limit-subgroup-ICE}
All ICE groups over Droms RAAGs are limit groups over Droms RAAGs. Moreover, a group is a limit group over Droms RAAGs if and only if it is a finitely generated subgroup of an ICE group over a Droms RAAG.
\end{theorem}

Theorem \ref{Limit-subgroup-ICE} implies the following.

\begin{lemma} \label{resolution} 
Let $\Gamma$ be a limit group over Droms RAAGs. Then, the trivial $\mathbb{Q}[\Gamma]$-module $\mathbb{Q}$ ( resp. the trivial $\mathbb{Z}[\Gamma]$-module $\mathbb{Z}$) has a free resolution with finitely generated modules and of finite length.
\end{lemma}

\begin{lemma} \label{Euler-char} \cite{preprint1} Let $G$ be a Droms RAAG and $\Gamma$ a limit group over $G$. Then, $\chi(\Gamma)\leq 0$. Furthermore, $\chi(\Gamma)=0$ if and only if
 $\Gamma$ has non-trivial center.
\end{lemma}

By \cite{Baumslag}, groups that are residually Droms RAAGs are precisely finitely generated subgroups of the direct product of limit groups over Droms RAAGs. Thus, the study of a group that is residually Droms RAAG reduces to the study of specific subdirect products. The study of homological properties of subdirect products of free groups was initiated by G. Baumslag and J. Roseblade in \cite{Roseblade}. Later on, the theory was developed by M. Bridson, J. Howie, C. F. Miller and H. Short in a sequence of papers that culminated in \cite{Bridson}, \cite{Bridson2}.
In \cite{Jone}, following  \cite{Bridson}, J. Lopez de Gamiz Zearra generalised  the results concerning subgroups of direct products of limit groups over free groups to the case of subgroups of direct products of  limit groups over Droms RAAGs.

\begin{theorem}\cite[Theorem 3.1]{Jone}\label{Theorem 3.1}
If $\Gamma_{1},\dots, \Gamma_{n}$ are limit groups over Droms RAAGs and $S$ is a subgroup of $\Gamma_{1}\times \cdots \times \Gamma_{n}$ of type $FP_{n}(\mathbb{Q})$, then $S$ is virtually a direct product of limit groups over Droms RAAGs.
\end{theorem}

\begin{theorem}\cite[Theorem 8.1]{Jone}\label{Theorem 8.1}
Let $\Gamma_{1},\dots,\Gamma_{n}$ be limit groups over Droms RAAGs where each $\Gamma_i$ has trivial center and let $S< \Gamma_{1}\times \cdots \times \Gamma_{n}$ be a finitely generated full subdirect product. Then either:\\[3pt]
(1) $S$ is of finite index; or\\[3pt]
(2) $S$ is of infinite index and has a finite index subgroup $S_{0}<S$ such that $H_{j}(S_{0}, \mathbb{Q})$ has infinite dimension for some $j\leq n$.
\end{theorem}

We recall that an embedding $S  \hookrightarrow \Gamma_{0}\times \cdots \times \Gamma_{n}$ of a finitely generated group that is residually Droms RAAG as a full subdirect product of limit groups over Droms RAAGs is \emph{neat} if $\Gamma_{0}$ is abelian (possibly trivial), $S \cap \Gamma_{0}$ is of finite index in $\Gamma_{0}$ and $\Gamma_{i}$ has trivial center for $i\in \{1,\dots,n\}$.

\begin{theorem}\cite[Theorem 10.1]{Jone} \label{Theorem 10.1}
Let $S$ be a finitely generated group that is residually Droms RAAG. The following are equivalent:\\[3pt]
(1) $S$ is finitely presentable;\\[3pt]
(2) $S$ is of type $FP_{2}(\mathbb{Q})$;\\[3pt]
(3) $\dim_{\mathbb{Q} } H_{2}(S_{0},\mathbb{Q})$ is finite for all subgroups $S_{0}<S$ of finite index;\\[3pt]
(4) there exists a neat embedding $S \hookrightarrow \Gamma_{0}\times \cdots \times \Gamma_{n}$ into a product of limit groups over Droms RAAGs such that the image of $S$ under the projection to $\Gamma_{i}\times \Gamma_{j}$ has finite index for $0\leq i<j \leq n$;\\[3pt]
(5) for every neat embedding $S \hookrightarrow \Gamma_{0}\times \cdots \times \Gamma_{n}$ into a product of limit groups over Droms RAAGs the image of $S$ under the projection to $\Gamma_{i}\times \Gamma_{j}$ has finite index for $0\leq i<j \leq n$.
\end{theorem}

The previous result is generalized in \cite{preprint1}, inspired by \cite[Theorem 11]{Desi}.

\begin{theorem} \cite{preprint1}
 Let $G_1, \dots,   G_m$ be limit groups over Droms RAAGs such that each $G_i$ has  trivial center and let \[S <  G_1 \times \cdots \times G_m\] be a finitely generated full subdirect product. Suppose further that, for some fixed number $s \in \{ 2, \dots, m \}$, for every subgroup $S_0$ of finite index in $S$ the homology group $H_i (S_0 , \mathbb{Q})$ is finite dimensional (over $\mathbb{Q}$) for all $ i \leq  s.$ Then for every canonical projection
\[p_{j_1, \dots,  j_s} \colon S \mapsto G_{j_1} \times \cdots \times G_{ j_ s}\]
the index of $p_{j_1, \dots, j_s} (S)$ in  $G_{j_1} \times \cdots \times G_{ j_s}$  is finite.

In particular, if $S$ is of type $FP_s$ over $\mathbb{Q}$, then for every canonical projection $p_{j_1, \dots,  j_s}$ the index of $p_{j_1, \dots, j_s} (S)$ in  $G_{j_1} \times \cdots \times G_{ j_s}$  is finite.
\end{theorem}

\section{Preliminaries on the Bieri--Neumann--Strebel--Renz invariants}

Let $G$ be a finitely generated group. By definition, the \emph{character sphere} $S(G)$ is \[S(G) = \frac{Hom (G , \mathbb{R}) \setminus \{ 0 \}}{\sim},\] where $\chi_1 \sim \chi_2$ for $\chi_1, \chi_2 \in Hom(G, \mathbb{R}) \backslash \{0\}$  if there is  $r > 0$ such that  $\chi_1 = r \chi_2$. The equivalence class of  $\chi_1$ is denoted by $[\chi_1]$. Note that $S(G)$ can be identified  with the unit sphere in $\mathbb{R}^{n}$, where $n$ is the torsion-free rank of the abelianization $G\slash G^\prime$. A \emph{  character} of $G$ is a non-trivial group homomorphism $$\chi \colon G \mapsto \mathbb{R} $$
 and $\chi$ is  \emph{discrete} if $\text{im}(\chi) \simeq \mathbb{Z}$.

Let $D$ be an integral domain and let $M$ be a (right)  $DG$-module. By definition, \[\Sigma^n_D(G, M) = \{[\chi] \in S(G) \mid M \textrm{ is of type $FP_n$ as a $DG_{\chi}$-module}\}.\]   When $M=D$ is the trivial $DG$-module, $\Sigma^n_D(G, M) = \Sigma^n_D(G, D)$ is usually denoted by $ \Sigma^n(G, D)$ and  we say that the invariant has coefficients in $D$. Also $\Sigma^n_{\mathbb{Z}}(G, M)$ is normally written as $\Sigma^n(G, M)$.

Suppose that $X$ is a finite generating set for $G$ and let $\Gamma$ be the \emph{Cayley graph} of $G$ with respect to $X$.  Then, the vertex set $V(\Gamma)$ equals $G$ and the edge set $E(\Gamma)$ is $X \times G$, where the edge $(x, g)$ has initial vertex $g$ and terminal vertex $xg$. 
The graph $\Gamma$ is equipped with an obvious $G$-action  given by $(x,g)\cdot h = (x, gh)$ for $ x \in X, g,h \in G$.
 
Let $\chi \colon G \mapsto \mathbb{R}$ be a character and let $\Gamma_{\chi}$ be the subgraph of $\Gamma$ spanned by $G_{\chi} = \{ g \in G \mid \chi(g) \geq 0\}$. Then, by definition, \[\Sigma^1(G) = \{ [\chi] \in S(G) \mid \Gamma_{\chi} \text{ is connected} \}.\]

If $G$ is finitely presented and $\langle X \mid R \rangle$ is a finite presentation of $G$, then $\Gamma$ can be extended to a $2$-complex, called the \emph{Cayley complex}, by gluing $2$-cells at every vertex attached at boundaries with labels that correspond to the elements of the set of relations  $R$. Then, if $\Gamma_{\chi} $ is the subcomplex of $\Gamma$ spanned by $G_{\chi}$, \[ \Sigma^2(G) = \{ [\chi] \in S(G) \mid \text{there exists a finite presentation } G = \langle X \mid R \rangle \text{ such that }\Gamma_{\chi} \text{ is 1-connected} \}.\]
It is important to note that even though the definition of $\Sigma^1$ is independent from the choice of a finite generating set $X$, in the definition of $\Sigma^2$ we cannot require that $\langle X \mid R \rangle$ is any finite presentation of $G$.

The invariant $\Sigma^1(G)$ coincides with $\Sigma^1(G, \mathbb{Z})$. This follows from \cite[(1.3)]{B-Renz}, where it is shown that $\Sigma^1(G, \mathbb{Z})$ coincides with the invariant defined in \cite{B-N-S}. The homotopical invariant $\Sigma^n(G)$ is defined only for groups of homotopical type $F_n$. The definition generalizes the definition of $\Sigma^2(G)$ by substituting the Cayley complex by the $n$-dimensional skeleton of a $K(G,1)$-complex and substituting 1-connectivity with $(n-1)$-connectivity. But since by \cite{Renzthesis} for $n \geq 2$  \[\Sigma^n(G) = \Sigma^2(G) \cap \Sigma^n(G, \mathbb{Z}),\] in order to calculate $\Sigma^n(G)$ it suffices to work with $\Sigma^2(G)$ and $\Sigma^n(G, \mathbb{Z})$.

The importance of the invariants $\Sigma^m(G)$ and $\Sigma^m(G, \mathbb{Z})$ lies in the fact that they control which subgroups of $G$ containing the commutator subgroup are of type $F_m$ and $FP_m$.
 
 \begin{theorem} \label{BRhomotopic} 
\leavevmode
\begin{itemize}
\item[(a)] \cite{Renzthesis} Let $G$ be a group of type $F_n$ and let $N$ be a subgroup of $G$ that contains the commutator subgroup $G^\prime$. Then, $N$ is of type $F_n$ if and only if
	\[S(G,N) = \{ [\chi] \in S(G) \mid \chi(N) = 0 \} \subseteq \Sigma^n(G).\]
	
\item[(b)] \cite{B-Renz}  Let $G$ be a group of type $FP_n$ and let $N$ be a subgroup of $G$ that contains the commutator subgroup $G^\prime$. Then, $N$ is of type $FP_n$ if and only if
	\[S(G,N) = \{ [\chi] \in S(G) \mid \chi(N) = 0 \} \subseteq \Sigma^n(G, \mathbb{Z}).\]
\end{itemize}
\end{theorem}
 
The following result is well-known. Let us denote the center of a group $G$ by $Z(G)$. Part a)  follows directly from \cite[Theorem C]{B-Renz} and part b) is \cite[Lemma 2.1]{Meinert-VanWyk2}.
 
\begin{lemma} \label{central}  Let $G$ be a group and let $\chi \colon G \mapsto \mathbb{R}$ be a character such that $\chi(c) \not= 0$ for some $c \in Z(G)$.\\[5pt]
(a) If $G$ is of type $FP_n$, then $[\chi] \in \Sigma^n(G, \mathbb{Z})$.\\[5pt]
(b) If $G$ is of type $F_n$, then $[\chi] \in \Sigma^n(G)$.
\end{lemma}

\section{Bieri--Neumann--Strebel--Renz $\Sigma$-invariants of limit groups over Droms RAAGs and their subdirect products}
\label{Section 7}

The aim of this section is to compute the Bieri--Neumann--Strebel--Renz invariants $\Sigma^m(G)$ and $\Sigma^m(G,\mathbb{Q})$ for a group $G$ that is a limit group over Droms RAAGs, and the invariant $\Sigma^1(H)$ for a group $H$ that is a finitely presented residually Droms RAAG. 

\begin{proposition}[= Proposition A] \label{sigma1} Let $G$ be a limit group over Droms RAAGs such that $G$ has a trivial center. Then $$\Sigma^n(G) = \Sigma^n(G, \mathbb{Z}) = \Sigma^n (G, \mathbb{Q}) = \emptyset \hbox{ for every }n \geq 1.$$
\end{proposition} 

\begin{proof} By \cite[Lemma 29]{Desi}, for every free-by-(torsion-free nilpotent) group $G$ such that $\chi(G) < 0$ and the
trivial $\mathbb{Q} G$-module $\mathbb{Q}$ has a free resolution of finite length and with finitely generated modules, we have that  $\Sigma^n (G, \mathbb{Q})$ and $\Sigma^n(G)$ are the empty set for every $n \geq 1$. Then, it is enough to apply Lemma \ref{resolution} and Lemma \ref{Euler-char} and the fact that for any group $H$ of type $F_n$ we have that $\Sigma^n (H) \subseteq \Sigma^n(H, \mathbb{Z} ) \subseteq \Sigma^n(H, \mathbb{Q})$.
\end{proof}

In \cite{Benno} Kuckuck suggested the following generalization of the VSP Criterion from \cite{Bridson2}. The Virtual Surjection Conjecture is known to hold for $n \leq 2$ (see \cite{Bridson2}).

{\bf The Virtual Surjection Conjecture}
{\it Let $n \leq m$ be positive integers and let $H$ be a subgroup of the direct product $G_1 \times \cdots \times G_m$ where $G_i$ is of type $F_n$ for $1 \leq i \leq m$. If $H$ is virtually surjective on $n$-tuples, then $H$ is of type $F_n$.} 

Following the above conjecture D. Kochloukova and F. Lima suggested in \cite{Desi2} the following conjecture.

{\bf The Monoidal Virtual Surjection Conjecture} {\it Let $n \geq m$ be positive integers. Let $H \leq L_1 \times \cdots \times L_m$ be a  full subdirect product of non-abelian limit groups over free groups $L_1, \dots, L_m$ such that $H$ is of type $FP_n$ and finitely presented. Then \[[\chi] \in \Sigma^n(H, \mathbb{Q}) = \Sigma^n(H, \mathbb{Z}) = \Sigma^n(H)\] if and only if 	
\begin{equation} \label{monod0} p_{j_1, \dots, j_n}(H_{\chi}) = p_{j_1, \dots, j_n}(H) \quad \text{for all} \quad 1 \leq j_1 < \dots  < j_n \leq m, \end{equation}
where $p_{j_1, \dots, j_n}\colon H \mapsto L_{j_1} \times \cdots \times L_{j_n}$ is  the canonical projection.}
 
In \cite[Theorem A]{Desi2} it was proved that the forward direction of the Monoidal Virtual Surjection Conjecture holds. Here we show that the same result holds even for full subdirect products of limit groups over Droms RAAGs with trivial center. In order to prove this, we will need the following result.

\begin{proposition} \label{limit1} Let $m \geq 2$, $1 \leq n \leq m$ and  $H < L_1 \times \cdots \times L_m$ be a finitely generated full subdirect product. Assume further that\\[5pt]
(a) $L_i$ is finitely generated and there is a free normal subgroup $F_i$ of $L_i$ such that  $L_i\slash F_i$ is polycyclic-by-finite for each $1 \leq i \leq m$;\\[5pt]
(b) $N \coloneqq F_1 \times \cdots \times F_m \subseteq H^\prime$;\\[5pt]
(c) for each $i$ there is a finite length free resolution of the trivial $\mathbb{Q} L_i$-module $\mathbb{Q}$ with all modules finitely generated and the Euler characteristic $\chi(L_i) < 0$;\\[5pt]
(d) $[\chi] \in \Sigma^n( H, \mathbb{Q})$.

Then for any $1 \leq j_1 < \dots < j_n \leq m$ we have that
\[\psi (H_{\chi}\slash N) = \psi(H\slash N)\]
where $\psi \colon H\slash N \mapsto L_{j_1}\slash F_{j_1} \times \cdots \times L_{j_n}\slash F_{j_n}$ is induced by $p_{j_1, \dots, j_n}$.
\end{proposition}

\begin{proof} By definition, $[\chi] \in \Sigma^n( H, \mathbb{Q})$ is equivalent to $\mathbb{Q}$ being of type $FP_n$ as a $\mathbb{Q} H_{\chi}$-module.
Let \[{\mathcal F} \colon \dots \mapsto F_i \mapsto F_{i-1} \mapsto \dots \mapsto F_0 \mapsto \mathbb{Q} \mapsto 0\]
be a free resolution of $\mathbb{Q}$ as a  $\mathbb{Q} H_{\chi}$-module with $F_i$ finitely generated for $i \leq n$.  Then ${\mathcal F} \otimes_{\mathbb{Q} N} \mathbb{Q}$ is a complex  whose modules $M_i \coloneqq F_i \otimes_{\mathbb{Q} N} \mathbb{Q}$ are free $\mathbb{Q} (H_{\chi}\slash N)$-modules and $M_i$ is finitely generated for each $i \leq n$.\\[5pt]
1) Suppose first that $\chi$ is a discrete character. Since $H\slash N$ is polycyclic-by-finite, we can apply \cite[Lemma~5.2]{Desi2} to get that $\mathbb{Q}  (H_{\chi}\slash N)$ is a Noetherian ring. Hence, for $ i \leq n$, we have that
$M_i$ is a Noetherian  $\mathbb{Q} (H_{\chi}\slash N)$-module. Then if $d_n \colon M_n\mapsto M_{n-1}$ is the differential of the complex  ${\mathcal F} \otimes_{\mathbb{Q} N} \mathbb{Q}$, we deduce that $\ker(d_n)$ is a finitely generated $\mathbb{Q} (H_{\chi}\slash N)$-module. Therefore,
\[H_n(N, \mathbb{Q}) \simeq H_n({\mathcal F} \otimes_{\mathbb{Q} N} \mathbb{Q}) = \ker(d_n)\slash \text{im}( d_{n+1})\] is finitely generated as a $\mathbb{Q} (H_{\chi}\slash N)$-module, so it is a Noetherian $\mathbb{Q} (H_{\chi}\slash N)$-module. On the other hand, by the Kunneth formula, \[H_n(N, \mathbb{Q}) \simeq \bigoplus_{1 \leq j_1 < \dots < j_n \leq m} ((F_{j_1}\slash [F_{j_1}, F_{j_1}]) \otimes_{\mathbb{Z} } \mathbb{Q} ) \otimes_{\mathbb{Q}} \dots  \otimes_{\mathbb{Q}} ((F_{j_n}\slash [F_{j_n}, F_{j_n}])  \otimes_{\mathbb{Z} } \mathbb{Q} )\] and each direct summand is a $\mathbb{Q} (H\slash N)$-submodule, where the $H\slash N$-action is induced by conjugation.

In particular, by Noetherianess, \[W_{j_1, \dots, j_n} \coloneqq ((F_{j_1}\slash[F_{j_1}, F_{j_1}]) \otimes_{\mathbb{Z} } \mathbb{Q} ) \otimes_{\mathbb{Q}} \dots  \otimes_{\mathbb{Q}} ((F_{j_n}\slash[F_{j_n}, F_{j_n}])  \otimes_{\mathbb{Z} } \mathbb{Q} )\] is a finitely generated $\mathbb{Q} (H_{\chi}\slash N)$-submodule of $H_n(N, \mathbb{Q})$. Since the action of $H_{\chi}\slash N$ on $W_{j_1, \dots, j_n}$ factors through $\psi (H_{\chi}\slash N)$, we deduce that \begin{equation} \label{sig} W_{j_1, \dots, j_n} \hbox{ is a finitely generated }\mathbb{Q} ( \psi (H_{\chi}\slash N) )\text{-module}. \end{equation}

Each $L_i\slash F_i$ is polycyclic-by-finite, so there is a characteristic subgroup $\widehat{Q}_i$  of finite index in $Q_i \coloneqq L_i\slash F_i$ that is torsion-free and polycyclic. Even though $\widehat{Q}_i$ is not nilpotent but polycyclic, the proof of \cite[Proposition 7]{Desi} applies in this case, that is, condition (c) from the statementent implies that each $F_{j_r}\slash [F_{j_r}, F_{j_r}]\otimes_{\mathbb{Z}} \mathbb{Q}$ contains a non-zero free cyclic $\mathbb{Q} (\widehat{Q}_i)$-submodule. In order words, $ \mathbb{Q} \widehat{Q}_i$ embeds in  $F_{j_r}\slash [F_{j_r}, F_{j_r}]\otimes_{\mathbb{Z}} \mathbb{Q}$. Then,
\begin{equation} \label{sig2} \mathbb{Q}[ \widehat{Q}_{j_1} \times \cdots \times \widehat{Q}_{j_n}]\hbox{ embeds in }W_{j_1, \dots, j_n }.\end{equation}

Let $\chi_0 \colon H\slash N \mapsto \mathbb{R}$ be the discrete character induced by $\chi \colon H \mapsto \mathbb{R}$. Then, there is $q_0 \in  H\slash N$ such that $\chi(q_0)>0$ and we have disjoint unions
\[ H \slash N= \dot{\bigcup}_{\alpha \in \mathbb{Z}} q_0^{\alpha} \ker(\chi_0) \quad \text{and} \quad H_{\chi} \slash N= \dot{\bigcup}_{\alpha \geq 0} q_0^{\alpha} \ker(\chi_0).\]
Applying $\psi$,
\[\psi( H \slash N)= \bigcup_{\alpha \in \mathbb{Z}} {\psi(q_0)}^{\alpha} \psi(\ker(\chi_0)) \quad \text{and} \quad \psi( H_{\chi} \slash N)= \bigcup_{\alpha \geq 0} {\psi(q_0)}^{\alpha} \psi(\ker(\chi_0)),\]
but these last two unions are not necessarily disjoint.

Suppose that $\psi(H_{\chi} \slash N) \not= \psi(H\slash N)$. Then $\psi(q_0)^{ -1} \not\in
\psi(H_{\chi}\slash N) $ and this implies that \[ [\psi(H\slash N) \colon \psi(\ker(\chi_0))] = \infty.\] Thus, $\psi(H\slash N) = \psi(\ker (\chi_0)) \rtimes \langle \psi(q_0) \rangle$ and there is a discrete character \[ \mu \colon  K \coloneqq \psi(H\slash N) \mapsto \mathbb{R}\] such that $\ker(\mu) = \psi(\ker ( \chi_0))$, $\mu(\psi(q_0)) = \chi_0 ( q_0)$.

Furthermore, note that $\psi(H_{\chi}\slash N) = K_{\mu} = \{ k \in K \mid \mu (k) \geq 0 \}$. Now (\ref{sig}) can be rewritten as $W_{j_1, \dots, j_n}$ is finitely generated as a $\mathbb{Q} K_{\mu}$-module, that is, $W_{j_1, \dots, j_n}$ is of type $FP_0$ as a $\mathbb{Q} K_{\mu}$-module. Hence, in the language of $\Sigma$-theory,  $[\mu] \in \Sigma^0_{\mathbb{Q}}(K, W_{j_1, \dots, j_n}).$

Let $K_1 \coloneqq K \cap (\widehat{Q}_{j_1} \times \cdots \times \widehat{Q}_{j_n})$ and $\mu_1 \coloneqq \mu_{| _{K_1}} \colon K_1 \mapsto \mathbb{R}$. Since $ \widehat{Q}_{j_1} \times \cdots \times \widehat{Q}_{j_n}$ has finite index in $Q_{j_1} \times \cdots \times Q_{j_n}$, we have that $[K \colon K_1] < \infty$. Then, by \cite[Theorem 9.3]{Meinert-VanWyk}, $[\mu_1] \in \Sigma^0_{\mathbb{Q}}(K_1, W_{j_1, \dots, j_n})$, that is, \begin{equation} \label{eqqe}  W_{j_1, \dots, j_n} \hbox{ is } FP_0 \hbox{ (that is, it is finitely generated) as a }\mathbb{Q} ( K_1) _{\mu_1}\hbox{-module.} \end{equation}  Note that \[\mathbb{Q} ( K_1) _{\mu_1} = \mathbb{Q} (K_{\mu} \cap  (\widehat{Q}_{j_1} \times \cdots \times \widehat{Q}_{j_n})),\] and since $\mu_1$ is a discrete character, $\mathbb{Q} ( K_1) _{\mu_1}$ is a Noetherian ring. The Noetherianess of  $\mathbb{Q} ( K_1) _{\mu_1}$, (\ref{sig2}) and (\ref{eqqe}) imply that 
 $ \mathbb{Q}[ \widehat{Q}_{j_1} \times \cdots \times \widehat{Q}_{j_n}]$ is finitely generated as a $\mathbb{Q} ( K_1) _{\mu_1}$-module.
 
This easily leads to a contradiction. Indeed, by the Noetherianess of  $\mathbb{Q} ( K_1) _{\mu_1}$ and the fact that $\mathbb{Q} K_1$ is a  $\mathbb{Q} ( K_1) _{\mu_1}$-submodule of $ \mathbb{Q}[ \widehat{Q}_{j_1} \times \cdots \times \widehat{Q}_{j_n}]$, we deduce that $\mathbb{Q} K_1$ is finitely generated as a
$\mathbb{Q} ( K_1) _{\mu_1}$-module, which is obviously false for a non-zero character $\mu_1$.\\[5pt]
2) The fact that the case of a discrete character implies the general case is the same as in the proof of \cite[Theorem A]{Desi2}.
\end{proof}

\begin{theorem}[= Theorem B] \label{applications1} Let  $m \geq 2$, $1 \leq n \leq m$ and $H < L_1 \times \cdots \times L_m$ be a finitely presented full subdirect product of limit groups over Droms RAAGs $L_1, \dots, L_m$  where each $L_i$ has trivial center. Suppose that $[\chi] \in \Sigma^n(H, \mathbb{Q})$. Then
\begin{equation} \label{monod0} p_{j_1, \dots, j_n}(H_{\chi}) = p_{j_1, \dots, j_n}(H) \quad \text{for all} \quad 1 \leq j_1 < \dots  < j_n \leq m, \end{equation} where $p_{j_1, \dots, j_n}\colon H \mapsto L_{j_1} \times \cdots \times L_{j_n}$ is  the canonical projection.
\end{theorem}
		
\begin{proof} It suffices to show that we can apply Proposition \ref{limit1}. Condition (c) from Proposition \ref{limit1} is Lemma \ref{Euler-char}. We claim that conditions (a) and (b) from Proposition \ref{limit1} hold. In the case when each $L_i$ is a non-abelian limit group over free groups the fact that conditions (a) and (b) from Proposition \ref{limit1} hold is established in \cite[Proposition 5.1]{Desi2}. The proof of \cite[Proposition 5.1]{Desi2} requires three properties of $L_i$:\\[5pt]
1) each $L_i$ is free-by-nilpotent;\\[5pt]
2) for $N_{i,j} \coloneqq p_j(\ker(p_i))$ we have that $[N_{1,j}, \dots, N_{j-1,j}, N_{j+1,j} , \dots, N_{m,j}] \subseteq H$;\\[5pt]
3) each $N_j \coloneqq \bigcap_{i \neq j}N_{i,j}$ has finite index in $L_j$.
	 
But if $L_i$ is a limit group over Droms RAAGs with trivial center, condition 1) is proved in \cite[Proposition A]{preprint1}, condition 2) is proved in \cite[Lemma 6.1]{Bridson} and for condition 3) see the argument before \cite[Theorem 6.2]{Jone} (the notation is different in \cite{Jone}).
\end{proof}

The next step is to compute the invariant $\Sigma^1(H)$ for a finitely presented residually Droms RAAG $H$. For that, we start with a technical result.
	
\begin{lemma} \label{converse}
Let  $H < L_1 \times \cdots \times L_m$ be a subdirect product. Assume further that\\[5pt]
(1) the group $L_i$ is finitely generated and  there is a free normal subgroup $F_i$ of $L_i$ such that $L_i\slash F_i$ is polycyclic-by-finite for each $1 \leq i \leq m$;\\[5pt]
(2) $N \coloneqq F_1 \times \cdots \times F_m \subseteq H^\prime$.	

Then \[\{ [\chi] \in S(H) \mid 
	p_i(H_{\chi}) = p_i(H) =  L_i \text{ for every }1 \leq i \leq m \} \subseteq \Sigma^1(H).\]
\end{lemma}
	
\begin{proof} 
Let us check that under the above conditions the group $H$ is finitely generated. Indeed, fix a finite subset $A_i \subseteq F_i$ such that $F_i = \langle A_i \rangle^{ L_i}$
and a finite subset $B_i \subseteq H$ such that $L_i = p_i(H) = \langle p_i(B_i) \rangle.$
Then, \[F_i = \langle A_i \rangle^{ L_i} = \langle A_i \rangle^{ p_i(H)} = \langle A_i \rangle^{ \langle p_i(B_i) \rangle} = \langle A_i \rangle^{ \langle B_i \rangle}  \subseteq \langle A_i \cup B_i  \rangle\]
and this implies that 
	\[N = F_1 \times \cdots \times F_m \subseteq \big{\langle} \bigcup_{ 1 \leq i \leq m} (A_i \cup B_i) \big{\rangle} \subseteq H.\]
In addition, $H\slash N$ is a subgroup of the polycyclic-by-finite group $L_1\slash F_1 \times \cdots \times L_m \slash F_m$,  hence $H\slash N$ is finitely generated, so there is a finite subset $C \subseteq H$ such that $H = N \langle C \rangle$. Thus, $C \cup \bigcup_{ 1 \leq i \leq m} (A_i \cup B_i)$ is a finite generating set for $H$.

By \cite{B-N-S}, we have that $[\chi] \in \Sigma^1(H) = \Sigma^1(H, \mathbb{Z})$ if there is a finitely generated  submonoid $M$ of $H_{\chi}$ such that $H^\prime$ is finitely generated as a $M$-group, where $M$ acts via conjugation (on the right). Therefore, we just need to show that if $[\chi]\in S(H)$ and $p_i(H_{\chi})=p_i(H)=L_i$ for every $1\leq i \leq m$, then $H^\prime$ is finitely generated as a $M$-group for some  finitely generated  submonoid $M$ of $H_{\chi}$. 

Since $p_i(H_{\chi}) = L_i $ and each $L_i$ is finitely generated, there is a finitely generated monoid $M$ such that $M \subseteq H_{\chi}$ and $p_i(M) = L_i $ for $ 1 \leq i \leq m$. Then,
\[F_i = \langle A_i \rangle^{ L_i} =  \langle A_i \rangle^{p_i(M)} \quad \text{and} \quad N = F_1 \times \cdots \times F_m \subseteq {\big{\langle} \bigcup_{ 1 \leq i \leq m} A_i {\big\rangle}}^M \subseteq H^\prime.\]
Finally, since $H^\prime \slash N$ is a subgroup of the polycyclic-by-finite group $H\slash N$, we deduce that it is finitely generated, so there is a finite subset $D$ of $H^\prime$ such that $H^\prime = N \langle D \rangle$. Then, \[H^\prime =   {\big{\langle}\big{(} \bigcup_{ 1 \leq i \leq m} A_i\big{)} \cup D \big{\rangle}}^M.\]
\end{proof} 
	
In \cite{Desi2} the following result is proved.
	
\begin{theorem} \cite[Theorem B]{Desi2} Let $H < L_1 \times \cdots \times L_m$ be a finitely presented full subdirect product of non-abelian limit groups over free groups $L_1, \dots, L_m$ with $m \geq 1$. Then  \[\Sigma^1(H) =  \Sigma^1(H, \mathbb{Q}) = \{ [\chi] \in S(H) \mid p_i(H_{\chi}) =  p_i(H) = L_i \text{ for every }1 \leq i \leq m \}.\]
\end{theorem}

We show that the above result works in a more general setting.
	
\begin{proposition}  \label{limit11} Let  $H < L_1 \times \cdots \times L_m$ be a  finitely generated full subdirect product. Assume further that\\[5pt]
(1) the group $L_i$ is finitely generated and there is a free normal subgroup $F_i$ of $L_i$ such that $L_i\slash F_i$ is polycyclic-by-finite for each $1 \leq i \leq m$;\\[5pt]
(2) $N \coloneqq F_1 \times \cdots \times F_m \subseteq H^\prime$;\\[5pt]	
(3) for each $i$ there is a finite length free resolution of the trivial $\mathbb{Q} L_i$-module $\mathbb{Q}$ with all modules finitely generated and $\chi(L_i) < 0$.
	
Then  \[\Sigma^1(H) =  \Sigma^1(H, \mathbb{Q}) = \{ [\chi] \in S(H) \mid p_i(H_{\chi}) =  p_i(H) =  L_i \text{ for every }1 \leq i \leq m \}.\]
\end{proposition}
	
\begin{proof} For a general group $G$ we have that $\Sigma^1(G) = \Sigma^1(G, \mathbb{Z}) \subseteq \Sigma^1(G, \mathbb{Q})$. Then, applying  Proposition \ref{limit1} for $n = 1$ we have that
	\[\Sigma^1(H) \subseteq  \Sigma^1(H, \mathbb{Q}) \subseteq \{ [\chi] \in S(H) \mid 
	p_i(H_{\chi}) = p_i(H) = L_i \text{ for every }1 \leq i \leq m \}.\]
The converse \[\{ [\chi] \in S(H) \mid p_i(H_{\chi}) = p_i(H) =  L_i \text{ for every }1 \leq i \leq m \} \subseteq \Sigma^1(H)\] is Lemma \ref{converse}.
\end{proof}

\begin{lemma} \label{ses} Let $H$ be a group with a normal subgroup $A$ of type $F_{\infty}$, let $\chi \colon H \mapsto \mathbb{R}$ be a non-zero character such that $\chi(A) = 0$  and let $\chi_0 \colon \overline{H} \coloneqq H\slash A \mapsto \mathbb{R}$ be the character induced by $\chi_0$.\\[5pt]
a)  If $H$ is of type $FP_n$, then $[\chi] \in  \Sigma^n(H, \mathbb{Z})$ if and only if $[\chi_0] \in \Sigma^n(\overline{H}, \mathbb{Z})$.\\[5pt]
b) If $H$ is of type $F_n$, then $[\chi] \in \Sigma^n(H)$ if and only if $[\chi_0] \in \Sigma^n(\overline{H})$.
\end{lemma} 

\begin{proof}
a) Consider the short exact sequence of groups
$1 \to A \to H \to \overline{H} \to 1$. Then, by \cite[Proposition 2.7]{Bieribook}, $H$ is of type $FP_n$ if and only if $\overline{H}$ is of type $FP_n$. The same argument as in  \cite[Proposition 2.7]{Bieribook} applied to the short exact sequence of monoids $1 \to A \to  H_{\chi} \to \overline{H}_{\chi_0} \to 1$ implies that $[\chi] \in \Sigma^n(H, \mathbb{Z})$ if and only if $[\chi_0] \in \Sigma^n(\overline{H}, \mathbb{Z})$.\\[5pt]
b) Since for a group $G$ of type $F_n$ with $n \geq 2$ we have that $\Sigma^n (G) = \Sigma^2(G) \cap \Sigma^n(G, \mathbb{Z})$ and $\Sigma^1(G, \mathbb{Z}) = \Sigma^1(G)$, in order to prove b) it suffices to consider only the case $n = 2$. But the case $n = 2$ is precisely \cite[Corollary 4.2]{Me}.
\end{proof} 
Finally, we compute $\Sigma^1(H)$.

\begin{corollary}[=Corollary C]\label{cor} a) Let $H < L_1 \times \cdots \times L_m$ be a finitely presented full subdirect product of limit groups over Droms RAAGs  $L_1, \dots, L_m$ where each $L_i$ has trivial center. Then  \[\Sigma^1(H) =  \Sigma^1(H, \mathbb{Q}) = \{ [\chi] \in S(H) \mid p_i(H_{\chi}) =  p_i(H) = L_i \text{ for every }1 \leq i \leq m \}\]
and thus,
\[S(H) \setminus \Sigma^1(H) = \bigcup_{1 \leq i \leq m} S(H, \ker (p_i)),\]
where $p_i \colon H \mapsto L_i$ is the canonical projection.\\[5pt]
b)  If $H$ is  a finitely presented residually  Droms RAAG, then there exist finitely many subgroups $H_1, \dots, H_m$ of $H$ such that
\[S(H) \setminus \Sigma^1(H) = \bigcup_{1 \leq i \leq m} S(H, H_i).\]
\end{corollary}
	
\begin{proof} a) The first equality is a corollary of Proposition \ref{limit11}. The fact that conditions (1), (2) and (3) from Proposition \ref{limit11} hold are established in the proof of Theorem \ref{applications1}.

By \cite[Lemma~5.10]{Desi2}, $ p_i(H_{\chi}) =  p_i(H)$ is equivalent to $\chi( \ker (p_i) )\not= 0$. Then, $[\chi] \in S(H) \setminus \Sigma^1(H)$ if and only if $ p_i(H_{\chi}) \not=  p_i(H)$ for some $i$. The latter is equivalent to $[\chi] \in S(H, \ker ( p_i))$.\\[5pt]
b) A finitely presented residually  Droms RAAG $H$ is a subgroup of a finite direct product of limit groups over Droms RAAGs, so $H < L_0 \times L_1 \times \cdots \times L_m$ is a full subdirect product where $L_0= \mathbb{Z}^{ k}$ for some $ k \geq 0$ and $L_1, \dots, L_m$ are limit groups over Droms RAAGs with trivial center for $1 \leq i \leq m$. Furthermore, we can assume that $L_0 \cap H$ has finite index in $L_0$. By construction $H \cap L_0$ is a central in $H$, so by Lemma \ref{central}, for every $n \geq 1$
\[ \{ [\chi] \in S(H) ~| ~\chi(H \cap L_0) \not= 0 \} \subseteq \Sigma^1(H). \]
Thus, if $[\chi] \in S(H) \setminus \Sigma^1(H)$, we have that $\chi(H \cap L_0) = 0$ and $\chi$ induces a character $$\chi_0 \colon \overline{H} \coloneqq H\slash (H \cap L_0) \mapsto \mathbb{R}.$$ 
Note that by Lemma \ref{ses},   \[ [\chi] \in S(H) \setminus \Sigma^1(H) \hbox{ if and only if }[\chi_0] \in S(\overline{H}) \setminus \Sigma^1(\overline{H}).\]
Since $ \overline{H} < L_1 \times \cdots \times L_m$ is a full subdirect product we can apply part a) to deduce that 
\[S(\overline{H}) \setminus \Sigma^1(\overline{H}) = \bigcup_{1 \leq i \leq m} S(\overline{H}, \ker (\overline{p}_i)),\]
where $\overline{p}_i \colon \overline{H} \mapsto L_i$ is the canonical projection.
Then, if we take $H_i$ to be the preimage of $\ker (\overline{p}_i)$ in $H$, we have that
\[S(H) \setminus \Sigma^1(H) = \bigcup_{1 \leq i \leq m} S(H, H_i).\]
\end{proof}

The last part of the section is related to proving that $\Sigma^2(H, \mathbb{Q})_{dis}= \Sigma^2(H, \mathbb{Z})_{dis}= \Sigma^2(H)$. In \cite{Desi2} it is shown that the Virtual Surjection Conjecture implies the discrete case of the Monoidal Virtual Surjection Conjecture:

\begin{theorem} \cite[Theorem D]{Desi2}
If the Virtual Surjection Conjecture holds in dimension $n$ and $\chi$ is a discrete character, then the Monoidal Virtual Surjection Conjecture holds for $\chi$. In particular, if $n = 2$ and $\chi$ is a discrete character, then the Monoidal Virtual Surjection Conjecture holds for $\chi$. Thus, \[\Sigma^2(H)_{dis} = \Sigma^2(H, \mathbb{Z})_{dis} = \Sigma^2(H, \mathbb{Q})_{dis}, \] where for $T \in \{ \Sigma^2(H, \mathbb{Z}), \Sigma^2(H, \mathbb{Q}), \Sigma^2(H) \}$ we write $T_{dis} = \{ [\chi] \in T \mid \chi \text{ is discrete} \}$.
\end{theorem}
 
We first generalize this theorem to the following result and then we apply it to our case.
 
\begin{theorem} \label{5-conditions}  Let $1 \leq n \leq m, m \geq 2$ and  $H < L_1 \times \cdots \times L_m$ be  a full subdirect product of type $F_n$. Assume further that\\[5pt]
(1) there is a free normal subgroup $F_i$ of $L_i$ such that $L_i\slash F_i$ is polycyclic-by-finite for each $1 \leq i \leq m$;\\[5pt]
(2) $N \coloneqq F_1 \times \cdots \times F_m \subseteq H^\prime$;\\[5pt]
(3) for each $i$ there is a finite length free resolution of the trivial $\mathbb{Q} L_i$-module $\mathbb{Q}$ with all modules finitely generated and $\chi(L_i) < 0$;\\[5pt]
(4) every finitely generated subgroup of $L_i$ is of type $F_n$ for each $1 \leq i \leq m$;\\[5pt]
(5) suppose that the Virtual Surjection Conjecture holds in dimension $n$.

Then \[[\chi] \in \Sigma^n(H, \mathbb{Q})_{dis} = \Sigma^n(H, \mathbb{Z})_{dis} = \Sigma^n(H)_{dis}\] if and only if 	
\begin{equation} \label{monod0} p_{j_1, \dots, j_n}(H_{\chi}) = p_{j_1, \dots, j_n}(H) \quad \text{for all} \quad 1 \leq j_1 < \dots  < j_n \leq m, \end{equation}
where $p_{j_1, \dots, j_n} \colon H \mapsto L_{j_1} \times \cdots \times L_{j_n}$ is  the canonical projection. In particular, since the Virtual Surjection Conjecture holds in dimension $2$, the result holds for $n = 2$ without further assumptions. \end{theorem}
	
\begin{proof} 
The conditions we impose here on $H, L_1, \dots , L_m$ are stronger than the ones imposed in Proposition \ref{limit1}. Suppose that $[\chi] \in \Sigma^n(H, \mathbb{Q})_{dis}$. Then, by Proposition \ref{limit1}, \[p_{j_1, \dots, j_n}(H_{\chi}) = p_{j_1, \dots, j_n}(H).\]
For the converse, assume that $\chi \colon H \mapsto \mathbb{R}$ is a non-zero discrete character and $p_{j_1, \dots, j_n}(H_{\chi}) = p_{j_1, \dots, j_n}(H) \quad \text{for all} \quad 1 \leq j_1 < \dots  < j_n \leq m$. Since in general we have that $\Sigma^n(H) \subseteq \Sigma^n(H, \mathbb{Z}) \subseteq \Sigma^n(H, \mathbb{Q})$, it suffices to show that
$[\chi] \in \Sigma^n(H)$. If we show that $N_0 \coloneqq \ker (\chi)$ is of type $F_n$, then by Theorem \ref{BRhomotopic} (a) we have that both $[\chi]$ and $[- \chi]$ belong to $\Sigma^n(H)$.

Note that by  Lemma \ref{converse} the result holds for $n = 1$. Furthermore, by \cite[Lemma~5.10]{Desi2} $p_{j}(H_{\chi}) = p_j(H)$ is equivalent to $\chi(Ker(p_j))  \not= 0$. This implies that  $p_{j}(H_{\chi}) = p_j(H)$ is equivalent to  $p_{j}(H_{-\chi}) = p_j(H)$. Then, applying  Lemma \ref{converse} again we have that $\{ [\chi], [- \chi] \} \subseteq \Sigma^1(H)$. Since $\chi$ is a discrete character we deduce by Theorem \ref{BRhomotopic} (a)  that
 $N_0$ is finitely generated.

 Let us write $H \coloneqq N_0 \rtimes \langle t \rangle$ with $\chi(t) > 0$. Since \[p_{j_1, \dots, j_n}(H) =  p_{j_1, \dots, j_n}(H_{\chi}) = \bigcup_{i \geq 0} p_{j_1, \dots, j_n}(N_0) p_{j_1, \dots, j_n}(t)^i\] is a group, $p_{j_1, \dots, j_n}(N_0)$ has finite index in $p_{j_1, \dots, j_n}(H)$. The group $H$ is of type $FP_n$, so by \cite[Theorem B]{preprint1}, $p_{j_1, \dots, j_n}(H)$ has finite index in $L_{j_1} \times \cdots \times L_{j_n}$. Hence, $ p_{j_1, \dots, j_n}(N_0)$ has also finite index in $L_{j_1} \times \cdots \times L_{j_n}$. In particular, $ p_{j_1, \dots, j_n}(N_0)$ has finite index in $p_{j_1}(N_0)  \times \cdots \times p_{j_n}(N_0)$. 

We consider $N_0$ as a finitely generated subdirect product of $p_1(N_0) \times \cdots \times p_m(N_0)$. The group $N_0$ is finitely generated, so $p_i(N_0)$ is a finitely generated subgroup of $L_i$. Then, by condition (4), $p_i(N_0)$ is of type $F_n$. Thus, if the Virtually Surjective Conjecture holds in dimension $n$ for the subdirect product $N_0 < p_1(N_0) \times \cdots \times p_m(N_0)$,  we deduce that $N_0$ is of type $F_n$.

Finally, the Virtual Surjection Conjecture holds for $n = 2$ (see \cite{Bridson2}).
\end{proof}
	
\begin{corollary}[= Theorem D] \label{applications12} 
a) Let $1\leq n \leq m$ and $m\geq 2$ be integers and $H < L_1 \times \cdots \times L_m$ be a  full subdirect product of limit groups over Droms RAAGs  $L_1, \dots, L_m$ such that each $L_i$ has trivial center and $H$ is of type $FP_n$ and finitely presented. Suppose that the Virtual Surjection Conjecture holds in dimension $n$. Then \[[\chi] \in \Sigma^n(H, \mathbb{Q})_{dis} = \Sigma^n(H, \mathbb{Z})_{dis} = \Sigma^n(H)_{dis}\] if and only if 	
\begin{equation} \label{monod0} p_{j_1, \dots, j_n}(H_{\chi}) = p_{j_1, \dots, j_n}(H) \quad \text{for all} \quad 1 \leq j_1 < \dots  < j_n \leq m, \end{equation} where $p_{j_1, \dots, j_n}\colon H \mapsto L_{j_1} \times \cdots \times L_{j_n}$ is  the canonical projection. In particular, since the  Virtual Surjection Conjecture holds in dimension $2$, the result holds for $n = 2$ without further assumptions.
 
b) If $H$ is  a finitely presented residually  Droms RAAG, then there exist finitely many subgroups $H_{i,j}$ of $H$, where $ 1 \leq i < j \leq n$,  such that
\[S(H)_{dis} \setminus \Sigma^2(H)_{dis} = \bigcup_{1 \leq i < j  \leq m} S(H, H_{i,j})_{dis}.\] 
\end{corollary}
	
\begin{proof} 
a) This follows from Theorem \ref{5-conditions}. Note that by \cite[Theorem B]{preprint1}, $p_{j_1, \dots, j_n}(H)$ has finite index  in $L_{j_1} \times \cdots \times L_{j_n}$ for every $1 \leq j_1 < \dots  < j_n \leq m$.\\[5pt]
b) By \cite{Baumslag} $H$ can be viewed as a full subdirect product of $L_0 \times L_1 \times \cdots \times L_m$, where $L_0= \mathbb{Z}^k$ for some $k \geq 0$, $L_0 \cap H$ has finite index in $L_0$ and $L_1, \dots, L_m$ are limit groups over Droms RAAGs such that each $L_i$ has trivial center for $1 \leq i \leq m$. Since $H \cap L_0$ is central in $H$, by Lemma \ref{central} for every $n \geq 1$,
\[ \{ [\chi] \in S(H) \mid \chi(H \cap L_0) \not= 0 \} \subseteq \Sigma^n(H). \]
Thus, if $[\chi] \in S(H) \setminus \Sigma^2(H)$ we have that $\chi(H \cap L_0) = 0$ and $\chi$ induces a character \[\chi_0 \colon \overline{H} \coloneqq H\slash (H \cap L_0) \mapsto \mathbb{R}.\] 
By Lemma \ref{ses},  \[ [\chi] \in S(H) \setminus \Sigma^2(H) \hbox{ if and only if }[\chi_0] \in S(\overline{H}) \setminus \Sigma^2(\overline{H}).\]

Since $ \overline{H} < L_1 \times \cdots \times L_m$, we can apply part a) to deduce that 
\[ [\chi_0] \in S(\overline{H})_{dis} \setminus \Sigma^2(\overline{H})_{dis} \hbox{ if and only if }
 p_{i,j} (\overline{H}) \not= p_{i,j}(\overline{H}_{\chi_0})  \hbox{ for some } 1 \leq i < j \leq m,\]
where $\overline{p}_{i,j} \colon \overline{H} \mapsto L_i \times L_j$ is the canonical projection.
By \cite[Lemma~5.10]{Desi2}, $  \overline{p}_{i,j}  (\overline{H}) \not=  \overline{p}_{i,j}  (\overline{H}_{\chi_0})$ is equivalent to $\chi_0( \ker(  \overline{p}_{i,j} )) = 0$. Then, we can define $H_{i,j}$ as the preimage of $\ker(   \overline{p}_{i,j}  )$ in $H$ and obtain that
\[S(H)_{dis} \setminus \Sigma^2(H)_{dis} = \bigcup_{1 \leq i \leq m} S(H, H_{i,j})_{dis}.\] 
\end{proof} 
	 
\section{More on finitely presented subdirect products} \label{sect-f}

We first state some known results on $\Sigma$-invariants related to graphs of groups and direct products.

\subsection{Some technical results on $\Sigma$-invariants : connections with actions on trees and $\Sigma$-invariants of direct products} 

\begin{theorem} \cite{Sch} \label{Bass-Serre-1}
Suppose that $G$ acts on a tree $T$ such that the quotient graph $T\slash G$ is finite and $\chi \colon G \mapsto \mathbb{R}$ is a non-zero character. Suppose further that the restriction $\chi_{\sigma} \colon G_{\sigma}  \mapsto \mathbb{R}$ of $\chi$ to the stabilizer $G_{\sigma}$ of a vertex or an edge $\sigma$ of $T$ is non-zero.\\[5pt]
i) If $n \geq  1,$ if $[\chi_v ] \in \Sigma^n (G_v , M )$ for all vertices $v$ of $T$ and if $[\chi_e ] \in \Sigma^{n-1} (G_e , M )$ for all edges $e$ of $T$, then $[\chi] \in \Sigma^n (G , M )$.\\[5pt] 
ii) If $n \geq 0,$ if $[\chi] \in \Sigma^n (G , M )$ and if $[\chi_e ] \in \Sigma^{n} (G_e , M )$ for all edges $e$ of $T$, then $[\chi_v ] \in \Sigma^n (G_v , M )$  for all vertices $v$ of $T$.\\[5pt]
iii) If $n \geq 1,$ if $[\chi] \in \Sigma^n (G , M )$ and if $[\chi_v ] \in \Sigma^{n-1} (G_v , M )$ for all vertices $v$ of $T$, then $[\chi_e ] \in \Sigma^{n-1} (G_e , M )$ for all edges $e$ of $T$.
\end{theorem}
	 
\begin{theorem} \cite[Theorem B]{Me} \label{Bass-Serre-2} Let a group $G$ act on a $1$-connected $G$-finite $2$-complex $X$. If $\chi \colon G \mapsto \mathbb{R}$ is a homomorphism such that $\chi_{\sigma} \not= 0$ and $[\chi_{\sigma}] \in \Sigma^{ 2 - dim(\sigma)} (G_{\sigma})$ for all cells $\sigma$ of $X$, then $[\chi] \in \Sigma^2(G)$. 
\end{theorem}

Recall that a graph of groups $\Gamma$ is said to be \emph{reduced} if, given an edge $e$ with distinct endpoints $v_1$ and $v_2$, the inclusions $G_e \hookrightarrow G_{v_{i}}$ are proper. In particular, an HNN extension is always reduced. An amalgam $G_1 \ast_{A} G_2$ is reduced if and only if it is non-trivial ($A \neq G_1, G_2$). If $\Gamma$ is not reduced, it can be made reduced by iteratively collapsing edges that make it non-reduced.

We say that a graph of groups is \emph{not an ascending HNN extension} if it is not an HNN extension (it has more than one edge or more than one vertex), or it is an HNN extension $G= \langle G_1, t \mid t^{-1} a t= \sigma(a) \text{ for all } a\in A \rangle$ but both $A$ and $\sigma(A)$ are proper subgroups of $G_1$.

\begin{theorem} \cite[Proposition 2.5]{Cashen} \label{Bass-Serre-3}
Let $G$ be the fundamental group of a finite reduced graph of groups $\Gamma$, with $G$ finitely generated. Assume that $\Gamma$ is not an ascending HNN extension. If $[\chi]\in \Sigma^1(G)$, then $\chi$ is non-trivial on every edge group.
\end{theorem}

Let $\chi = (\chi_1, \chi_2) \colon G_1 \times G_2 \mapsto \mathbb{R}$ be a non-zero real character, where by $\chi_i$ we denote the restriction of $\chi$ to $G_i$, $i\in \{1,2\}$. Let $\pi_i$ be the canonical projection $ G_1 \times G_2 \mapsto G_i$. Then, $\pi_i$ induces \[\pi_i^{\ast} \colon S(G_i) \mapsto S( G_1 \times G_2).\]
 
 \begin{lemma} \cite[Lemma 9.1]{Ge} Let $\chi_1 \colon G_1 \mapsto \mathbb{R}$ and $\chi_2 \colon G_2 \mapsto \mathbb{R}$ be two non-zero real characters and define $\chi$ to be $( \chi_1, \chi_2) \colon G \coloneqq G_1 \times G_2 \mapsto \mathbb{R}$.\\[5pt]
a) Suppose that $G_1$ and $G_2$ are groups of type $F_{n+ m + 1}$. If $[\chi_1] \in \Sigma^n(G_1)$ and $[\chi_2] \in \Sigma^m( G_2)$, then $[\chi] \in \Sigma^{ n+ m + 1}(G)$.\\[5pt]
b) Suppose that $G_1$ and $G_2$ are groups of type $FP_{n+ m + 1}$. If  $[\chi_1] \in \Sigma^n(G_1, \mathbb{Z})$ and $[\chi_2] \in \Sigma^m( G_2, \mathbb{Z})$, then $[\chi] \in \Sigma^{ n+ m + 1}(G, \mathbb{Z})$.
\end{lemma}
 
 \begin{lemma} \cite[Lemma 9.2]{Ge} \label{zero-direct-product} Let $\chi_i \colon G_i \mapsto \mathbb{R}$ be a non-zero real character and let $G$ be $G_1 \times G_2$.\\[5pt]
a) Suppose that $G$ is of type $F_n$. Then $[\chi_i] \in \Sigma^n (G_i)$ if and only if $\pi_i^{\ast}([\chi_i]) \in \Sigma^n ( G_1 \times G_2)$.\\[5pt]
b) Suppose that $G$ is of type $FP_n$. Then $[\chi_i] \in \Sigma^n (G_i, \mathbb{Z})$ if and only if $\pi_i^{\ast}([\chi_i]) \in \Sigma^n ( G_1 \times G_2, \mathbb{Z})$.
\end{lemma}
	 	 
The above two lemmas were used in \cite{Ge} to prove the following result.
	 
\begin{theorem} \cite[Corollary 2]{Ge}\label{direct-product}  Let $G_1, \dots, G_m$ be groups of type $F_{\infty}$ and suppose that $\Sigma^1(G_i) = \Sigma^{ \infty}(G_i)$ for each $i\in \{1,\dots,m\}$. Let $\chi = ( \chi_1, \dots, \chi_m) \colon G \coloneqq G_1 \times \cdots \times G_m \mapsto \mathbb{R}$ be a non-zero real character with $\chi_{i_1}, \dots, \chi_{i_k}$  the only non-zero characters  among $\chi_1, \dots, \chi_m$. Then,
	 \[ [\chi] \notin \Sigma^n(G) \iff \quad n \geq k \quad \text{  and } \quad [\chi_{i_r} ] \notin \Sigma^{1}(G_{i_r}) \hbox{ for } 1 \leq r \leq k.\]
\end{theorem}
	 
\subsection{New applications} 

We now apply the previous results to solve the question posed in \cite{Jone-Montse} and explained at the end of the introduction.

\begin{definition}\label{Definition hierarchy}
Let $\mathcal{K}$ be a class of groups. The \emph{$Z \ast$-closure of $\mathcal{K}$}, denoted by $Z \ast (\mathcal{K})$, is the union of classes $(Z \ast (\mathcal{K}))_{j}$ defined as follows. At level $0$, the class $(Z \ast (\mathcal{K}))_{0}$ equals $\mathcal{K}$. A group $K$ lies in $(Z \ast (\mathcal{K}))_{j}$ if and only if
\[ K \simeq \mathbb{Z}^m \times (G_{1}\ast \cdots \ast G_{n}),\] where $m\in \mathbb{N}\cup \{0\}$ and the group $G_{i}$ lies in $(Z \ast (\mathcal{K}))_{j-1}$ for all $i\in \{1,\dots,n\}$.

The \emph{level of $G$}, denoted by $l(G)$, is the smallest $j$ for which $G$ belongs to  $(Z \ast (\mathcal{K}))_{j}$.
\end{definition}



Let $\mathcal{C}$ be the class of fundamental groups of finite reduced graphs of groups where the vertex groups are free abelian of rank greater than $1$ and the edge groups are infinite cyclic or trivial. Let $\mathcal {J}$ be the \emph{$Z \ast$-closure of $\mathcal{C}$}.
	 
\begin{proposition}
Let $G \in \mathcal{J}$ and $\chi \colon G \mapsto \mathbb{R}$ be a non-zero real character. Then, $ \Sigma^1(G) = \Sigma^{\infty}(G).$
\end{proposition}
	 
\begin{proof}Since in general $ \Sigma^{\infty}(G) \subseteq \Sigma^1(G)$,  we have to show that $\Sigma^1(G) \subseteq \Sigma^{\infty}(G).$ Let $[\chi] \in \Sigma^1(G)$.
	 
We use induction on the level of $G$. Suppose that $G\in \mathcal{J}_{0}$. That is, $G \in \mathcal{C}$. Then, by Theorem \ref{Bass-Serre-3}, the restriction of $\chi$ to any edge group is non-zero, so in particular, the restriction of $\chi$ to every vertex group is non-zero. Since for every finitely generated free abelian group $A$ we have that $S(A) = \Sigma^{ \infty}(A)$, by the definition of $G$ and Theorem \ref{Bass-Serre-1} and Theorem \ref{Bass-Serre-2} we get that $[\chi] \in \Sigma^{ \infty} (G, \mathbb{Z})$ and $[\chi] \in \Sigma^2(G)$. Finally, since $\Sigma^{ \infty}(G) = \Sigma^2(G) \cap \Sigma^{ \infty} (G, \mathbb{Z})$ we deduce that $[\chi] \in \Sigma^{\infty}(G).$
	 
Suppose now that $G$ has level $k \geq 1$ and that the result holds for groups in $\mathcal{J}_{k-1}$. Then, \[G = \mathbb{Z}^m \times (G_1 \ast \dots \ast G_n),\] where each $G_i$ has level at most $k-1$.

If $m=0$, from Theorem \ref{Bass-Serre-3} we get that $n=1$, so the result follows from the inductive hypothesis. If $m\geq 1$,  assume that there is $c\in \mathbb{Z}^m$ such that $\chi(c)\neq 0$. Then, by Lemma \ref{central}  $[\chi] \in \Sigma^{\infty} (G)$. Let us deal with the case when $\chi |_{\mathbb{Z} ^m} = 0$. We set $\mu$ to be the restriction of $\chi$ to $G_1 \ast \dots \ast  G_n$. Since $[\chi] \in \Sigma^1(G)$, then $[\mu] \in \Sigma^1(G_1 \ast \dots \ast G_n)$. Again by  Theorem \ref{Bass-Serre-3}, $n= 1$ and $[\mu] \in \Sigma^{ \infty} (G_1)$. But since  $\chi |_{\mathbb{Z} ^m} = 0$ this is equivalent to $[\chi] \in \Sigma^{\infty} (G)$.
\end{proof}

\begin{theorem}[=Proposition E]\label{fin-pres-sigma} Let $S < H_1 \times H_2$ be a co-abelian, finitely presented, subdirect product with $H_1, H_2 \in \mathcal{J}$. Then, $S$ is of type $F_{\infty}$.
\end{theorem}
	 
\begin{proof}
Let $\chi = (\chi_1, \chi_2) \colon H  \coloneqq H_1 \times H_2 \mapsto   \mathbb{R}$ be a non-zero real character such that $\chi(S) = 0$. Since $S$ is finitely presented, by Theorem \ref{BRhomotopic}  $[\chi] \in \Sigma^2(H)$. We aim to show that $[\chi] \in \Sigma^{ \infty}(H)$ and hence, again by Theorem \ref{BRhomotopic}, $S$ is of type $F_{\infty}$.\\[5pt]	 
1) Suppose first that $\chi_1 \not= 0$ and  $\chi_2 \not= 0$. If $[\chi_1] \notin \Sigma^1( H_1) $ and $ [\chi_2] \notin \Sigma^1( H_2)$, then by Theorem \ref{direct-product} $[\chi] \notin \Sigma^2(H)$, which is a contradiction. Hence, for at least one $i$ we have that $[\chi_i] \in \Sigma^1(H_i) = \Sigma^{ \infty}(H_i)$. Then, by Theorem \ref{direct-product}, $[\chi] \in \Sigma^n(H)$ for every $n$, that is, $[\chi] \in \Sigma^{ \infty}(H)$.\\[5pt]	 
2) Suppose that $\chi_1$ or $\chi_2$ is the zero character. Without loss of generality, assume that $\chi_1 = 0$. Then, by Lemma \ref{zero-direct-product}, $[\chi] \in \Sigma^n(H) \iff [\chi_2] \in \Sigma^n(H_2)$. Since $[\chi] \in \Sigma^2(H)$ we have that $[\chi_2] \in \Sigma^2(H_2) \subseteq \Sigma^1(H_2) = \Sigma^{ \infty}(H_2)$. Hence, using again Lemma \ref{zero-direct-product}, $[\chi] \in \Sigma^{ \infty} (H)$.
\end{proof}

Finally, we can apply this theorem to the case in which we are interested. Recall that
the class $\mathcal{G}$ is defined to be the class of finitely generated cyclic subgroup separable graphs of groups with free abelian vertex groups of rank greater than $1$ and infinite cyclic or trivial edge groups. The class $\mathcal{A}$ is the $Z\ast-$closure of $\mathcal{G}$. 

\begin{corollary} \label{aplication} Let $S < H_1 \times H_2$ be a co-abelian, finitely presented, subdirect product with $H_1,H_2\in \mathcal{A}$. Then, $S$ is of type $F_{\infty}$.
\end{corollary}

\begin{proof} Since $\mathcal{A} \subseteq \mathcal{J}$, we can apply Theorem \ref{fin-pres-sigma}.
\end{proof}


\begin{thebibliography}{99}

\bibitem{Kisnney}
{\sc K. Almeida}, 
{\it  The BNS-invariant for Artin groups of circuit rank $2$ },
Journal of Group Theory, 2 (2018), 21, 189 -- 228.

\bibitem{Baumslag}
{\sc G.~Baumslag, A.~Miasnikov, V.~Remeslennikov,}
{\it Algebraic geometry over groups I. Algebraic sets and ideal theory,}
Journal of Algebra, (1999), 219,  16--79.

\bibitem{Roseblade}
{\sc G.~Baumslag, J-E.~Roseblade,}
{\it Subgroups of Direct Products of Free Groups,}
Journal of the London Mathematical Society, 1 (1984), 30, 44--52.

\bibitem{Bieribook}
{\sc R. Bieri,}
 {\it Homological Dimension of Discrete Groups,}
 2nd edition, Queen Mary College Mathematics Notes, 1981.

\bibitem{B-G-K}
{\sc R. Bieri, R. Geoghegan, D. Kochloukova,}
{\it  The Sigma Invariants of Thompson's Group $F$},
Groups, Geometry, and Dynamics, 2 (2010), 4, 263 --273.

\bibitem{B-Groves} {\sc R. Bieri, J. Groves,}
{\it The geometry of the set of characters induced by valuations},
Journal f\"ur die reine und angewandte Mathematik, (1984), 347, 168 --195.

\bibitem{B-N-S} 
{\sc R. Bieri, W. D. Neumann, R. Strebel,}
{\it A geometric invariant of discrete groups,}
Inventiones mathematicae, 3 (1987), 90, 451 -- 477.

\bibitem{B-Renz} 
{\sc R. Bieri, B. Renz,}
{\it Valuations on free resolutions and higher geometric invariants of groups,}
Commentarii Mathematici Helvetici, 3 (1988), 63, 464-- 497.

\bibitem{Conchita}
{\sc R. Blasco-Garc\'ia, J. I. Cogolludo-Agust\'in, C. Mart\'inez-P\'erez,}
{\it On the Sigma invariants of even Artin groups of FC-type,}
\url{https://arxiv.org/abs/2011.07608}.

\bibitem{Bridson} 
{\sc M. R. Bridson, J. Howie, C. F. Miller, H. Short,}
{\it Subgroups of direct products of limit groups,}
Annals of Mathematics,  3 (2009), 170, 1447 -- 1467.

\bibitem{Bridson2} 
{\sc M. R. Bridson, J. Howie, C. F. Miller, H. Short,}
{\it On the finite presentation of subdirect products and the nature of residually free groups,}
American Journal of Mathematics, 4 (2013), 135, 891 -- 933.

\bibitem{Montse}  
{\sc M. Casals-Ruiz, A. Duncan, I. Kazachkov,}
{\it Limit groups over coherent right-angled Artin groups,}
\url{https://arxiv.org/abs/2009.01899}.

\bibitem{Montse2} 
{\sc M. Casals-Ruiz, I. Kazachkov,}
{\it  Limit groups over partially commutative groups and group actions on real cubings,}
Geometry and Topology, 2 (2015), 19, 725--852.

\bibitem{Jone-Montse} 
{\sc M. Casals-Ruiz, J. Lopez de Gamiz Zearra,}
{\it Subgroups of direct products of graphs of groups with free abelian vertex groups,}
\url{https://arxiv.org/abs/2010.10414}.

\bibitem{Cashen}
{\sc C. H. Cashen, G. Levitt,}
{\it Mapping tori of free group automorphisms, and the Bieri–-Neumann–-Strebel invariant of graphs of groups,}
Journal of Group Theory, 2 (2016), 19, 191--216.


\bibitem{D}
{\sc T. Delzant,}
{\it L'invariant de Bieri–-Neumann–-Strebel des groupes fondamentaux des vari\'et\'es k\"ahl\'eriennes,}  Mathematische Annalen, (2010), 348, 119 --125.

\bibitem{Droms}
{\sc C.~Droms,}
{\it Graph groups, coherence, and three-manifolds,}
Journal of Algebra, 2 (1987), 106, 484--489.

\bibitem{Droms2}
{\sc C.~Droms,}
{\it Subgroups of Graph Groups,}
Journal of Algebra, 2 (1987), 110, 519--522.

\bibitem{Servatius} 
{\sc C. Droms, B. Servatius, H. Servatius,}
{\it Surface subgroups of graph groups,}
Proceedings of the American Mathematical Society, 3 (1989), 106, 573--578.
 
\bibitem{B-F}
{\sc F. Funke, D. Kielak},
{\it Alexander and Thurston norms, and the Bieri-Neumann-Strebel invariants for free-by-cyclic groups,}
Geometry \& Topology, 5 (2018), 22, 2647 -- 2696.
 
\bibitem{Ge}
{\sc R. Gehrke,} 
{\it The higher geometric invariants for groups with sufficient commutativity,}
Communications in Algebra, 4 (1998), 26, 1097 -- 1115.
 
\bibitem{DK}
{\sc D. Kielak,}
{\it The Bieri--Neumann--Strebel invariants via Newton polytopes},
Inventiones mathematicae, (2020), 219, 1009 -- 1068.
 
\bibitem{Desi} 
{\sc D. H. Kochloukova,}
{\it On subdirect products of type $FP_{m}$ of limit groups,}
Journal of Group Theory, 1 (2009), 13, 1--19.

\bibitem{Desi2} 
{\sc D. H. Kochloukova, F. Ferreira Lima,}
{\it On the Bieri--Neumann--Strebel--Renz invariants of residually free groups,}
to appear in Proceedings of the Edinburgh Mathematical Society.

\bibitem{K}
{\sc D. H. Kochloukova,}
{\it  On the Bieri--Neumann--Strebel--Renz $\Sigma^1$-invariant of even Artin groups}, accepted in Pacific Journal of Mathematics, \url{https://arxiv.org/abs/2009.14269}.

\bibitem{preprint1}
{\sc D. H. Kochloukova, J. Lopez de Gamiz Zearra,}
{\it On subdirect products of type $FP_n$ of limit groups over Droms RAAGs,}
\url{https://arxiv.org/abs/2104.14849}.

\bibitem{Benno}
{\sc B. Kuckuck,}
{\it Subdirect products of groups and the $n$-$(n+1)$-$(n+2)$ conjecture,}
The Quarterly Journal of Mathematics, 4 (2014), 65,  1293 -- 1318.

\bibitem{Jone} 
{\sc J. Lopez de Gamiz Zearra,}
{\it Subgroups of direct products of limit groups over Droms RAAGs,}
\url{https://arxiv.org/abs/2008.08382}.

\bibitem{Me}
{\sc H. Meinert,} 
{\it  Actions on $2$-complexes and the homotopical invariant $\Sigma^2$ of a group,}
Journal of Pure and Applied Algebra, 3 (1997), 119, 297 -- 317.

\bibitem{Meinert-VanWyk} 
{\sc J. Meier, H. Meinert, L. VanWyk,}
{\it Higher generation subgroup sets and the $\Sigma$-invariants of graph groups,}
Commentarii Mathematici Helvetici, 1 (1998), 73, 22 -- 44.

\bibitem{Meinert-VanWyk2}
{\sc J. Meier, H. Meinert, L. VanWyk,}
{\it On the $\Sigma$-invariants of Artin groups,}
Topology and its Applications, 1 (2001), 110, 71–-81.

\bibitem{Renzthesis} {\sc B. Renz,}
{\it Geometrische Invarianten und Endlichkeitseigenschaften von Gruppen,} 
PhD Thesis, Johann Wolfgang Goethe-Universit\"at Frankfurt am Main, 1988.

\bibitem{Sch}
{\sc S. Schmitt,}
{\it  \"Uber den Zusammenhang der geometrischen Invarianten von Gruppe und Untergruppe
mit Hilfe von variablen Modulkoeffzienten,} Diplomarbeit, Frankfurt a.M., 1991

\bibitem{W-Z}
{\sc  S. Witzel, M. Zaremsky,}
{\it The $\Sigma$-invariants of Thompson’s group F via Morse theory,}
Topological Methods in Group Theory, London Mathematical Society Lecture Note Series, Cambridge University Press, (2018), 251, 173--194.

\bibitem{Z}
{\sc M. Zaremsky,}
{\it On the $\Sigma$-invariants of generalized Thompson groups and Houghton groups,}
International Mathematics Research Notices, 19 (2017), 2017, 5861 -- 5896.

\end{thebibliography}
\end{document}